\documentclass[11pt]{amsart}
\usepackage{amsmath,amssymb,amsfonts,enumerate,graphicx,bm,fullpage,soul}
\usepackage{amsthm}
\usepackage{hyperref}

\newcommand{\FR}[1]{\medskip

  \noindent
  \framebox{\parbox{\dimexpr\linewidth-2\fboxsep-2\fboxrule}{#1}}\medskip

}

\newtheorem{Theorem}{Theorem}[section]

\newtheorem{Lemma}[Theorem]{Lemma}
\newtheorem{Computation}[Theorem]{Computation}
\newtheorem{Proposition}[Theorem]{Proposition}
\theoremstyle{definition} \newtheorem{Definition}[Theorem]{Definition}
\newtheorem{Example}[Theorem]{Example}
\newtheorem{Remark}[Theorem]{Remark}

\numberwithin{equation}{section}

\DeclareMathOperator{\coker}{Coker}%
\DeclareMathOperator{\codim}{codim}%
\DeclareMathOperator{\reg}{reg}%
\newcommand*{\defeq}{\mathrel{\vcenter{\baselineskip0.5ex
      \lineskiplimit0pt \hbox{\scriptsize.}\hbox{\scriptsize.}}}
  =}%

\def \QQ {\mathbb{Q}}%
\usepackage{multicol}

\newcommand{\llar}{-\kern-5pt-\kern-5pt\longrightarrow}

\def\restr{{\kern-1pt\restriction\kern-1pt}}

\title{Large Lower Bounds for the Betti Numbers of Graded
  Modules with Low Regularity}

\author{Adam Boocher}%
\address{A.\ Boocher,
  University of San Diego, San Diego, California, USA}
\email{aboocher@sandiego.edu and aboocher@gmail.com}

\author{Derrick Wigglesworth} %
\address{ D.\ Wigglesworth, University of Arkansas, Fayetteville,
  Arkansas, U.S.A.}  \email{drwiggle@uark.edu and
  derrick.wigglesworth@gmail.com}

\date{\today}

\begin{document}
\maketitle

\begin{abstract}Suppose that $M$ is a finitely-generated graded module
  (generated in degree $0$) of codimension $c\geq 3$ over a polynomial
  ring and that the regularity of $M$ is at most $2a-2$ where
  $a\geq 2$ is the minimal degree of a first syzygy of $M$.  Then we
  show that the sum of the betti numbers of $M$ is at least
  $\beta_0(M)(2^c + 2^{c-1})$.  Additionally, under the same
  hypothesis on the regularity, we establish the surprising fact that
  if $c \geq 9$ then the first half of the betti numbers are each at
  least twice the bound predicted by the Buchsbaum-Eisenbud-Horrocks
  rank conjecture: for $1\leq i \leq \frac{c+1}{2}$,
  $\beta_i(M) \geq 2\beta_0(M){c \choose i}$.
\end{abstract}

\section{Introduction}

Let $S = k[x_1,\ldots,x_n]$ be a polynomial ring over a field $k$ and
let $M$ be a finitely generated graded $S$-module of finite length.
The total betti number
$\beta(M) \defeq \beta_0(M) + \cdots + \beta_n(M)$ is defined to be
the sum of the betti numbers of $M$.  This number has been of recent
interest, most notably in the context of the Total Rank Conjecture
which predicts that $\beta(M) \geq 2^n$.  If $\mathrm{char}(k)\neq 2$,
this conjecture was recently proved by Walker \cite{W}, who also
showed that equality holds if and only if $M$ is isomorphic to $S$
modulo a regular sequence -- such modules are called complete
intersections.

Evidently if $M$ is not a complete intersection, then $\beta(M) > 2^n$
and since $\beta(M)$ must be even, it follows that
$\beta(M)\geq 2^n +2$.  In fact, there is reason to believe that if
$M$ is not a complete intersection then $\beta(M)$ must be
considerably larger than $2^n$.  It was asked by Charalambous, Evans,
and Miller in \cite{CEM} whether it is true that
$\beta(M) \geq 2^n +2^{n-1}$.  They proved that this is the case for
arbitrary graded modules $M$ when $n \leq 4$ and for all $n$ when $M$
is multi-graded.  We remark that if $M$ is not of finite length, then
the natural extension is to claim that
\begin{equation}\label{strong bound eq}
  \beta(M) \geq 2^c + 2^{c-1} \ \  \mbox{where $c$ is the codimension of $M$.}
\end{equation}  
Such an extension has recently been obtained for monomial ideals in
\cite{BS} where it was also proved that equality is possible for all
$c \geq 2$.  The aim of the present paper is to prove that
(\ref{strong bound eq}) holds for arbitrary $M$ provided that the
regularity of $M$ is small relative to the degrees of its first
syzygies.
\begin{Theorem}\label{MainTheorem}
  Let $M$ be a graded $S$-module of codimension $c\geq 3$ generated in
  degree $0$ and let $a\geq 2$ be the minimal degree of a first syzygy
  of $M$.  If $\reg(M) \leq 2a-2$, then
$$\beta(M) \geq \beta_0(M)(2^c+2^{c-1}).$$
\end{Theorem}

Our result is an extension of work by Erman \cite{Erman}, where he
proved, under the same hypothesis on the regularity, that
$\beta_i(M) \geq \beta_0(M){c \choose i}.$
Erman's work proves a special case of the Buchsbaum-Eisenbud-Horrocks
rank conjecture which states that $\beta_i(M)\geq {c \choose i}$.
Naturally, Erman's bound will imply that
$\beta(M) \geq \beta_0(M) 2^c$ when the regularity hypothesis holds.
Noting that $2^c + 2^{c-1} = (1.5)(2^{c})$, the stronger bound in
Theorem \ref{MainTheorem} asserts that on average, each betti number
$\beta_i(M)$ is at least $1.5$ times $\beta_0(M) {c \choose i}$.  We
achieve this bound by showing that except in a small number of cases
(which arise with $c \leq 8$) it is true that the first half of the
betti numbers are at least $2\beta_0(M){ c\choose i}$.
\begin{Theorem}\label{corollaryIntro}
  Let $M$ be a graded $S$-module of codimension $c\geq 9$ generated in
  degree $0$ and let $a\geq 2$ be the minimal degree of a first syzygy
  of $M$.  If $\reg(M) \leq 2a-2$ then for each
  $1\leq i\leq \lceil c/2\rceil$,
  $\beta_i(M)\geq 2\beta_0(M){c \choose i}$.
\end{Theorem}
This result implies a rather strong connection between the regularity
of $M$ and its first few betti numbers.
In the Artinian (finite-length) case, since the regularity can be
interpreted as the socle degree, we can understand this result as
making more precise the idea that having a small number of generators
will naturally lead to a high socle degree.  Our theorem provides
bounds on this relationship which are new (even in the Artinian case).

As an example, consider the following statements for quadrics: Suppose
$I$ is an ideal generated by quadrics of codimension $c \geq 9$.  If
$I$ has precisely $c$ minimal generators then $S/I$ is a complete
intersection with regularity $c$.  On the other hand, the ideal
$I = (x_1, \ldots, x_c)^2$ has ${c+1 \choose 2}$ minimal generators
and then $S/I$ has regularity $1$.  Theorem \ref{corollaryIntro}
implies that for the regularity of $S/I$ to drop below $3$, $I$ must
have at least $2c$ minimal generators.

It seems to us rather bizarre that this theorem (like Erman's results)
should depend almost completely on the numerics coming from
Boij-S\"oderberg Theory.  This mysterious behavior is also apparent in
McCullough's work in \cite{McC} concerning the relationship between
the regularity of an ideal and the degrees of half of its syzygies.
In this vein, our results can be interpreted as saying that the degree
of the first syzygy and the number of syzygies in the first half of
the resolution can in some cases force the regularity to be large.  We
remark that the regularity bound is actually relaxed enough to include
many interesting geometric examples.  In \cite{Erman}, Erman presents
several examples of modules that satisfy $\reg(M) \leq 2a - 2$
including smooth curves embedded by linear systems of high degree,
toric surfaces, and Artinian rings $M = S/I$ whose socle degree is
relatively low.

We comment now on our methods and how they differ from those of Erman.
We begin as he did with standard Boij-S\"oderberg techniques to write
an arbitrary betti diagram as a rational combination of normalized
\emph{pure} betti diagrams, whose entries $\pi_i(D)$ are each a
function of $n$ positive integers $D = (d_1, \ldots, d_n)$.  In
sections 2 and 3 we show that the proofs of our main theorems reduce
to finding lower bounds for $\pi_i(D)$.  Like Erman we reduce these
calculations to the study of a function $F(a,b,e,n,i)$ of $5$
variables.  It is here that our analysis differs substantially from
that of Erman.

Since Erman was concerned with a uniform bound for all betti numbers,
his proof (in our notation) shows that $F(a,b,e,n,i) \geq 1$.  As we
mentioned above, our main strategy to prove Theorem \ref{MainTheorem}
is to focus on the first half of the betti numbers and prove that they
are at least twice the bound that Erman proved.  Roughly speaking we
then want to show that $F(a,b,e,n,i) \geq 2$ for small $i$.  Since
this statement is not true for all $i$ (nor is it true if the
codimension $n$ is less than $9$) our analysis necessarily proceeds in
a delicate way.  In addition, if $n\leq 8$, since Theorem
\ref{MainTheorem} holds whereas \ref{corollaryIntro} does not,
independent techniques are developed to address these cases.  What
ultimately makes the proofs difficult is that even if one fixes all
but one variable, it is not necessarily the case that $F$ is an
increasing function, and thus finding its minimum requires some
care. Moreover, there are a whole host of cases where our general
method fails -- these arise primarily when the difference between the
regularity of $M$ and the generating degree of a first syzygy of $M$
is very small.  The reduction via Boij-S\"oderberg theory necessitates
that we consider all of these cases, as otherwise our results would be
significantly weaker.  These special cases complicate the structure of
our proof as evidenced by the flowchart (Fig.\ \ref{fig:FlowN9}) which
demonstrates how all the pieces fit together.

\section{Boij-S\"oderberg Basics}
\label{BS-Basics}
In this section we will review the relevant pieces of Boij-S\"oderberg
theory.  Rather than state the theory in its fullest generality, we
present only the version we need for our results.  We begin with an
example.

\begin{Example}
  Let $S = \QQ[x,y,z]$ and take $I$ to be an ideal generated by $5$
  random quadrics.  Set $M = S/I$.  Similarly, let $\phi$ be a
  $3\times 10$ matrix of random quadrics and let $N = \coker \phi$.
  Finally, let $M' = S/(x^2,y^2,z^2,xy)$.  The betti diagrams of $M,N$
  and $M'$ are given below:
  \begin{equation*}
    \begin{array}{c|cccc}
      \beta(M) & &&& \\
      \hline 
               & 1 & - & - & - \\ 
               & - & 5 & 5 & - \\ 
               & - & - & - & 1 \\ 
    \end{array}, \ \
    \begin{array}{c|cccc}
      \beta(N) & &&& \\
      \hline 
               & 3 & - & - & - \\ 
               & - & 10 & - & - \\ 
               & - & - & 15 & 8 \\ 
    \end{array},
    \begin{array}{c|cccc}
      \beta(M') & &&& \\
      \hline 
                & 1 & - & - & - \\ 
                & - & 4 & 2 & - \\ 
                & - & - & 3 & 2 \\ 
    \end{array}
  \end{equation*}
  We point out that the first two diagrams are \textbf{pure} in the
  sense that each column has at most one nonzero entry.  The last
  betti diagram is not pure since the column representing the second
  syzygy module has two nonzero entries.  Further, note that each of
  the first two diagrams is a \textbf{sub-diagram} of the third
  diagram, in the sense that the locations of the nonzero entries of
  the first two fit inside the third diagram.  This will be made
  explicit in what follows.

  Finally, we notice the rather astonishing fact that the third betti
  diagram (thought of as a matrix) can be written as a positive
  rational linear combination of the first two diagrams:
$$\beta(M') = \frac25 \beta(M) + \frac15 \beta(N).$$
The above example is an instance of the following, which is a summary
of the main results in Boij-S\"oderberg Theory.

\FR{ { \centering ``The betti diagram of an (arbitrary) finite-length
    module \\ can be written as a positive rational linear combination
    of pure diagrams.''

  } }
\end{Example}

\noindent We now set $S = k[x_1, \ldots, x_n]$ and work with finitely
generated graded $S$-modules $M$. Henceforth all of our modules will
be assumed to be generated in degree $0$; allowing for shifting, this
is tantamount to saying that $M$ is generated in a single degree.  If
$M$ is a finite length module and each syzygy module of $M$ is
generated in a single degree then we will say that $M$ \textbf{has a
  pure resolution} (or that $M$ \textbf{is pure}).  Note that we
require pure modules have finite length.  For a pure module $M$ we let
$D\colon (d_0=0) < d_1 < \cdots < d_n$ be the sequence whose $i$-th
entry is the degree of the generators of the $i$-th syzygy module of
$M$.  This increasing sequence of integers $D$ is called the
\textbf{degree sequence of $M$}.  By $\reg(D)$ we will mean the number
$d_n - n$, which corresponds to the regularity of the module $M$.

\begin{Remark}
  A finite length module $M$ is pure with degree sequence
  $D\colon (d_0=0) < d_1 < \cdots < d_n$ if and only if for each
  $i=0, \ldots, n$, the graded betti numbers of $M$ satisfy
$$\beta_{ij}(M) \neq 0 \iff j = d_i.$$
\end{Remark}

Remarkably, the betti numbers of pure modules are determined up to
scalar multiple.  Indeed, if a finite length module $M$ is pure with
degree sequence $D$ then there is a scalar $\lambda\in \QQ$ so that
for all $i$, the following holds:
\begin{equation}\label{TheHerzogKuhlEquations}
  \beta_i(M)
  = \beta_{i,d_i}
  =  \lambda  \pi_i(D)
  \text{ with }
  \pi_i(D) = \frac{d_1\cdots d_n}{\prod_{i\neq j} (d_i-d_j)}.
\end{equation}
This was first proven by Herzog and K\"uhl \cite{HK} and the
equalities above are called the Herzog-K\"uhl equations.  Note that
since $\pi_0(D) = 1$ we have that $\lambda= \beta_0(M)$. In order to
prove Theorems \ref{MainTheorem} and \ref{corollaryIntro} we will
study the rational functions $\pi_i$ and establish the following two
theorems.

\begin{Theorem}\label{PureDiagramBound}
  Suppose that $n\geq 3$ and $D\colon 0 < d_1< \ldots <d_n$ is a
  degree sequence of length $n+1$ with $d_1\geq 2$ satisfying
  $\reg(D) \leq 2d_1 - 2$.  Then $\sum \pi_i(D) \geq 2^n + 2^{n-1}.$
\end{Theorem}

\begin{Theorem}\label{FirstHalfAreDoubleWithCases}
  Let $D$ be a degree sequence of length $n+1$ with $d_1\geq 2$ and
  $\reg(D)\leq 2d_1 - 2$.
  
  \begin{itemize} \item If $n \geq 9$ then for each
    $1\leq i\leq \lceil n/2\rceil$,
  $$\pi_i(D)\geq 2\textstyle{n \choose i}.$$

\item If $n\in\{6,7,8\}$, the same conclusion holds unless
  \begin{itemize} \renewcommand{\labelitemii}{$\circ$}
  \item $d_1 = 2$ and $\reg(D) = 2$ or
  \item $d_1 = 3$ and $\reg(D) = 3.$
  \end{itemize}
\end{itemize}
 
\end{Theorem}

\begin{Remark}\label{The 36 Special Cases}
  When $n\in\{6,7,8\}$ there are only 36 degree sequences satisfying
  the regularity hypothesis but to which Theorem
  \ref{FirstHalfAreDoubleWithCases} does not apply.  The pure diagrams
  are those that are subdiagrams of one of the following diagrams:
$$\begin{array}{c|ccccccccc}
    & 0& 1& 2& 3& & & &  & n  \\
    \hline
    0 &\star&-&-&-&-&-&-&-&- \\
    1 &-& \star &\star &\star& \cdots &\cdots &\star&\star& -\\
    2 &-&-  &\star &\star&\cdots &\cdots&\star&\star&\star \\
    &&&&&&&&& \\
  \end{array}  \ \ \ \ \ , \ \ \ \ \ 
  \begin{array}{c|ccccccccc}
    & 0& 1& 2& 3& & & &  & n  \\
    \hline
    0 &\star&-&-&-&-&-&-&-&- \\
    1 &-&-&-&-&-&-&-&-&- \\
    2 &-& \star &\star &\star& \cdots &\cdots &\star&\star& -\\
    3 &-&-  &\star &\star&\cdots &\cdots&\star&\star&\star \\
  \end{array}.$$

\end{Remark}

The content of Theorems \ref{PureDiagramBound} and
\ref{FirstHalfAreDoubleWithCases} is purely numerical. Their
connection to our main theorems on betti numbers is achieved via the
beautiful results of Boij-S\"oderberg Theory, developed in
\cite{ES,BoijSoed}.  This theory shows that the betti diagram of an
arbitrary finite length module can be written as a finite rational
linear combination of pure diagrams.

Given a module $M$, its betti numbers $\beta_{ij}(M)$ are often
arranged into a \textbf{betti-diagram} -- thought of as a matrix
(typically with the convention that $\beta_{i,i+j}(M)$ is in the $i$th
column and the $j$th row).  With this convention the regularity of $M$
is equal to the index of the bottom row in the diagram.
If $D$ is a degree sequence of length $n+1$ then we define $B(D)$ to
be the betti diagram with entry $\pi_i(D)$ in column $i$ and row
$d_i+i$.  By the Herzog-K\"uhl equations
\eqref{TheHerzogKuhlEquations}, if $M$ is a pure module with degree
sequence $D$ then the betti diagram of $M$ will be a scalar multiple
of $B(D)$.

%
\begin{Example}
  We associate to the degree sequence $D=\{0,2,4,5\}$ the following
  diagrams:
$$
\begin{array}{c|cccc}
  & &&&\\
  \hline 
  & \star & - & -  & - \\ 
  & - & \star & -  & -   \\ 
  & - & -  & \star & \star  \\ 
\end{array}, \ \  
\begin{array}{c|cccc}
  B(D) & &&&\\
  \hline 
       & 1 & - & -  & - \\ 
       & - & \frac{10}{3} & - & - 
  \\ & - & - & 5 & \frac{8}{3}  \\ 
\end{array}.
$$
We use stars to emphasize that we care about the positions of the
nonzero entries in the diagram, then use $B(D)$ to denote the diagram
of numbers $\pi_i(D)$.
\end{Example}

Given two diagrams $B$ and $B'$ we say that $B'$ is a {\bf
  sub-diagram} of $B$ if for each nonzero entry of $B'$, the
corresponding entry in $B$ is also nonzero.  If $B$ is the betti
diagram of a finitely generated module then there are a finite number
of degree sequences $D$ such that $B(D)$ is a subdiagram of $B$.  We
now summarize the results of Eisenbud-Schreyer and (respectively)
Boij-S\"oderberg \cite{ES, BoijSoed} which show that a finite length
module (respectively, one of codimension $c$) can be decomposed as a
sum of pure diagrams.

\begin{Theorem}[Main Theorem of Boij-S\"oderberg Theory \cite{ES,
    BoijSoed}] \label{BSThmES} Let $M$ be a finitely generated
  $S$-module with betti diagram $B$.  Suppose that $\codim M = c$.  If
  $\Omega = \{ B(D) \}$ is the set of all pure sub-diagrams of $B$
  having between $c+1$ and $n+1$ columns (indexed by their degree
  sequences $D$ with lengths between $c+1$ and $n+1$) then there exist
  non-negative rational numbers $\lambda_D$ such that
$$B = \sum_{B(D) \in \Omega} \lambda_D B(D).$$ 
In particular, this implies that $\beta_0 = \sum \lambda_D$ and more
generally, $\displaystyle{\beta_i(M) = \sum \lambda_D \pi_i(D).}$
\end{Theorem}


\section{Reduction to Theorem \ref{FirstHalfAreDoubleWithCases}}
In this section we explain how to deduce our main theorems from their
numerical versions stated in Section \ref{BS-Basics}.  We will then
assume Theorem \ref{FirstHalfAreDoubleWithCases} and use it to prove
Theorem \ref{PureDiagramBound}.  For convenience, all four theorems
are restated in the diagram below.

\FR{
  \begin{multicols}{2}
    \begin{center}
      \textsc{Main Theorems on Betti Numbers}
    \end{center}
    \textbf{Theorem \ref{MainTheorem}.} Let $M$ be a graded $S$-module
    of codimension $c\geq 3$ generated in degree $0$ and let $a\geq 2$
    be the minimal degree of a first syzygy of $M$.  If
    $\reg(M) \leq 2a-2$ then $\beta(M) \geq \beta_0(M)(2^c+2^{c-1}).$
    \vspace{.2cm}

    \textbf{Theorem \ref{corollaryIntro}.} Let $M$ be a graded
    $S$-module of codimension $c\geq 9$ generated in degree $0$ and
    let $a\geq 2$ be the minimal degree of a first syzygy of $M$.  If
    $\reg(M) \leq 2a-2$ then for each $1\leq i\leq \lceil c/2\rceil$,
    $\beta_i(M)\geq 2{c \choose i}$.  \columnbreak

  \begin{center}
    \textsc{Main Numerical Results}
  \end{center}
  
  \textbf{Theorem \ref{PureDiagramBound}.} Suppose that $n\geq
  3$, 
  and $D$ is a degree sequence of length $n+1$, and $d_1\geq 2$
  satisfying $\reg(D) \leq 2d_1 - 2$.  Then
  $\sum \pi_i(D) \geq 2^n + 2^{n-1}.$ \vspace{.2cm}
  
  \textbf{Theorem \ref{FirstHalfAreDoubleWithCases}.} If $d_1\geq 2$
  and $\reg(D)\leq 2d_1-2$ and $n\geq 9$ then for each
  $1\leq i\leq \lceil n/2\rceil$,
  $$\textstyle{\pi_i(D)\geq 2{n \choose i}.}$$  If $n\in\{6,7,8\}$ and either
  $d_1\geq 3$ or $\reg(D) - d_1 +1 \neq 1$, then the same conclusion
  holds.

\end{multicols}
}

The theorems on the left follow more or less immediately from the
corresponding theorems on the right via Boij-S\"oderberg theory.  With
the exception of a small number of special cases when $n < 9$, Theorem
\ref{PureDiagramBound} will follow from Theorem
\ref{FirstHalfAreDoubleWithCases}, the proof of which will be
postponed until Section \ref{ProofOfFirstHalfAreDoubleWithCases}.

\subsection{Proofs of Theorems \ref{MainTheorem} and
  \ref{corollaryIntro}}
\begin{proof}[Proof of Theorem \ref{MainTheorem}]
  Suppose $M$ is generated in degree zero, and $a\geq 2$ is the
  minimal degree of a first syzygy of $M$.  By Theorem \ref{BSThmES}
  there exist nonnegative rational numbers $a_D$ such that
  \begin{equation} \label{BSSum} \beta_i(M) = \sum_D a_D \pi_i(D)
  \end{equation}
  where $D$ runs over all degree sequences of length
  $\ell(D) \in [c+1,n+1]$ whose betti diagrams, $B(D)$, are
  sub-diagrams of $B(M)$.  Let $D$ be such a degree sequence. Then
  $d_1\geq a$ and as we have assumed $\reg M \leq 2a-2$, it follows
  that
  $$\reg(D) = d_{\ell(D)} - \ell(D) \leq \reg M \leq 2a - 2 \leq 2d_1 - 2.$$
  Hence we can apply Theorem \ref{PureDiagramBound}.  Since every
  degree sequence appearing in the sum has length at least $c+1$,
  Theorem \ref{PureDiagramBound} implies that
  $\sum_i \pi_i(D) \geq 2^{c} + 2^{c-1}$.  Hence we have
  \begin{equation*}
    \beta(M)
    = \sum_{i=0}^n \beta_i(M)
    = \sum_D a_D \left(\sum_{i=0}^n  \pi_i(D)\right)
    \geq \sum_D a_D  (2^c + 2^{c-1})
    = \beta_0(M)(2^c + 2^{c-1}).
    \qedhere
  \end{equation*}
\end{proof}

\begin{proof}[Proof of Theorem \ref{corollaryIntro}]
  The scaffolding is exactly the same as in the previous proof.  If
  $c\geq 9$ then equation \eqref{BSSum} and Theorem
  \ref{FirstHalfAreDoubleWithCases} imply for
  $i \in \{1, \ldots, \lceil c/2 \rceil\}$
  \begin{equation*}
    \beta_i(M)
    = \sum_D a_D \pi_i(D)
    \geq \sum_D a_D \ 2{ c \choose i}
    =\beta_0(M){ c \choose i}. \qedhere
  \end{equation*}
\end{proof}

\subsection{Proof of Theorem \ref{PureDiagramBound}}
\begin{proof}[Proof of Theorem \ref{PureDiagramBound} when Theorem
  \ref{FirstHalfAreDoubleWithCases} holds]
  Suppose that $D$ is a degree sequence satisfying the hypotheses of
  Theorem \ref{FirstHalfAreDoubleWithCases}.  Then let us add up all
  of the $\pi_i$ in pairs.  If $n$ is odd, there are an even number of
  $\pi_i$'s.  When summing, we can group them in pairs
  $\pi_i + \pi_{n-i}$.  Now $\pi_0 + \pi_n \geq 2$ since $\pi_0 = 1$
  and $\pi_n \geq 1$ by Erman's Theorem.  In all other pairs, we
  combine Theorem \ref{FirstHalfAreDoubleWithCases} with Erman's
  result, and conclude that $\pi_i + \pi_{n-i} \geq 3{n \choose i}$.
  Moreover, since the assumption on indices in Theorem
  \ref{FirstHalfAreDoubleWithCases} includes $i=\lceil n/2\rceil$, the
  last pair is at least $4{n \choose (n-1)/2}$.  Thus
  \begin{equation*}
    \sum \pi_i
    \geq 2 + 3{n \choose 1} + \cdots+ 3{n \choose {\frac{n-1}{2}}} + {n \choose \frac{n-1}{2}}
    \geq 2 + \frac{3}{2}(2^n - 2) + {n \choose \frac{n-1}{2}}
    \geq 2^n + 2^{n-1}.
  \end{equation*}

  When $n$ is even, we proceed by pairing terms exactly as before.  In
  this case however, there is a central term in the sum (the term
  $\pi_{n/2}$) which has no companion.  We thus have:
  \begin{equation*}
    \sum \pi_i
    \geq 2 + 3{n \choose 1} + \cdots 3{n \choose {\frac{n-2}{2}}} + 2{ n \choose \frac{n}{2}}
    \geq 2 + \frac{3}{2}\left(2^n - 2 - {n \choose {\frac{n}{2}}}\right) + 2{n \choose \frac{n}{2}}
    \geq 2^n + 2^{n-1}.\qedhere
  \end{equation*}
\end{proof}

\begin{proof}[Proof of Theorem \ref{PureDiagramBound} for
  $n\in\{6,7,8\}$]
  By Remark \ref{The 36 Special Cases} there are only 36 degree
  sequences $D$ that satisfy $d_1 \geq 2$ and $\reg(D) \leq 2d_1 -2$
  for which Theorem \ref{FirstHalfAreDoubleWithCases} does not apply.
  Using Macaulay2 we checked that the sum of $\pi_i(D)$ in each of
  these cases is at least $2^n + 2^{n-1}$.  The reader is directed to
  the file \texttt{computations.m2} included in our arXiv posting for
  explicit code that can be used to verify this statement.
\end{proof}

\begin{proof}[Proof of Theorem \ref{PureDiagramBound} for
  $n\in\{3,4,5\}$]
  For each value of $n$, we will verify that
  $\sum \pi_i\geq 1.5\cdot 2^n$ via a direct computation.  Suppose
  first that $n=3$ so that the degree sequence $D=\{0,d_1,d_2,d_3\}$.
  We change notation to emphasize the nonlinear parts of $D$ by
  instead writing it as $D=\{0,a,a+x+1,a+x+y+2\}$, where $x,y\geq 0$
  can easily be computed from the $d_i$'s.  We may assume $a\geq 2$
  and our regularity assumption says $ x + y + 1 \leq a$.  We want to
  prove that $\sum_{i=0}^3\pi_i(D)\geq 12$.  Using the Herzog-K\"{u}hl
  equations, this is equivalent to the polynomial inequality
  \begin{equation*}
    a^2+ax+ay+2a-5xy-5x-5y-5\geq 0.
  \end{equation*}
  If $x=y=0$ so that the resolution is linear, then the assumption
  that $a\geq 2$ implies the inequality holds.  On the other hand if
  the resolution is not linear, we observe that the left hand side is
  clearly an increasing function of $a$, so it suffices to consider
  the case that $a=x+y+1$, whereby the inequality becomes
  \begin{equation*}
    0\leq 2x^2+2y^2-2xy-2=(x-y)^2+x^2+y^2+xy-2
  \end{equation*}
  Evidently, each of these terms is positive at least two are nonzero
  (since $x$ and $y$ are not both $0$), so the inequality holds as
  desired.

  Repeating an identical analysis with $n=4$ (so that
  $D=\{0,a,a+x+1, a+x+y+2,a+x+y+z+3\}$) again results in a polynomial
  inequality for which the left hand side is an increasing function of
  $a$.  After considering the linear case separately, we set
  $a=x+y+z+1$, and are left to verify the polynomial inequality
  \begin{multline*}
    2x^4+5x^3y+4x^2y^2+xy^3+7x^3z+9x^2yz+4xy^2z+2y^3z+9x^2z^2+8xyz^2\\
    +5y^2z^2+5xz^3+4yz^3+z^4+12x^3+19x^2y+10xy^2+3y^3+27x^2z+15xyz+12y^2z\\
    +23xz^2+17yz^2+8z^3+22x^2+13xy+9y^2+23xz+12yz+17z^2+6x+4z-6\geq 0
  \end{multline*}
  This will hold provided not all of $x,y,z=0$.

  The proof strategy for $n=5$ is exactly the same and begins by
  setting $D=\{0,a,a+x+1, a+x+y+2,a+x+y+z+3, a+x+y+z+w+4\}$, then
  using the Herzog-K\"uhl equations to get a polynomial inequality.
  The expression thus obtained is now too complicated to be analyzed
  by hand, though it's still very manageable for a machine.  By
  writing it as a polynomial in $a$, one can verify that all of the
  coefficients (besides the constant term) are positive and therefore
  that left hand side is increasing as a function of $a$.  Again
  substituting $a=x+y+z+w+1$, one obtains an expression and factors it
  (with a computer) to arrive at an inequality in which all terms on
  the left hand side are positive except for the constant term.  A
  simple computer verification shows that the inequality it holds for
  all $x,y,z,w\geq 0$.
\end{proof}

\begin{Remark}
  The file \texttt{computations.m2} included in our arXiv posting
  contains code to verify the numerical statements in this paper.
\end{Remark}

\section{Proof of Theorem \ref{FirstHalfAreDoubleWithCases}}
\label{ProofOfFirstHalfAreDoubleWithCases}
In this section we will prove Theorem
\ref{FirstHalfAreDoubleWithCases}, which is the last ingredient needed
to complete the proofs of our main results.  We endeavor to show that
for suitable $D$ and $i$, we have
\begin{equation*}
  \textstyle\pi_i(D) \geq 2 {n \choose i}.
\end{equation*}

Thus it is natural to study the function
$(D,i)\mapsto\pi_i(D)/{n \choose i}$.  Of course this function depends
on $n+1$ parameters, so a simplification is required before a
reasonable analysis can be performed.  We will define a function $F$
depending on five parameters such that
\begin{equation*}
  \frac{\pi_i(D)}{{n \choose i}}\geq F(a,b,e,n,i).
\end{equation*}

\noindent {\bf Main Notation:} Let $D\colon 0< d_1< \cdots < d_n$ and
set $a = d_1$.  Given $i\geq 1$, we define a modification of $D$ as
follows:
\begin{equation}\label{Di-definition}
  D^i
  = \{0, a, a+ 1, a+2, \ldots, a+(i-2), d_i, d_n - (n-i-1),  d_n - (n-i-2), \ldots d_n\}
\end{equation}

\noindent 
Considering now a degree sequence,
$D^i$ 
we will focus our attention on its nonlinear parts.
\begin{equation*}
  b = d_i - a - i + 1, \quad e = d_n -d_i - n + i.
\end{equation*}
Notice then that we have
\begin{align*}
  d_n &= a  + b + e + n - 1, \\
  \reg(D) &=  a + b + e - 1.
\end{align*}

The reader is urged to ignore these equations and press on to the
example that follows, which should clarify the idea (and resolve the
ambiguity when $i=1$).
\begin{Example} \label{ex:Di} Suppose that $i=5$ and
  $D = \{0, 3,5,6,8,\boxed{10}, 12, 15, 16, 19, 20\}$ then the betti
  diagrams for $D$ and $D^5$ would be formatted as shown
  \begin{equation*}
    \begin{array}{c|ccccccccccc}
      D \\ 
      \hline
      & \star & -     & -     & -     & -     & -      & -     & -     & -     & -     & - \\
      &       & -     & -     & -     & -     & -      & -     & -     & -     & -     & - \\ 
      &       & \star & -     & -     & -     & -      & -     & -     & -     & -     & - \\ 
      &       & -     & \star & \star & -     & -      & -     & -     & -     & -     & - \\ 
      &       & -     & -     & -     & \star & -      & -     & -     & -     & -     & - \\ 
      &       & -     & -     & -     & -     & \boxed{\star}& - & -   & -     & -     & - \\ 
      &       & -     & -     & -     & -     & -      & \star & -     & -     & -     & - \\ 
      &       & -     & -     & -     & -     & -      & -     & -     & -     & -     & - \\ 
      &       & -     & -     & -     & -     & -      & -     & \star & \star & -     & - \\ 
      &       & -     & -     & -     & -     & -      & -     & -     & -     & -     & - \\ 
      &       & -     & -     & -     & -     & -      & -     & -     & -     & \star & \star \\ 
    \end{array}, \ \
    \begin{array}{c|ccccccccccc}
      D^5 \\ 
      \hline 
      & \star & -    & -    & -     &-      & -      & -      &-      &-     & -      & - \\ 
      &       & -    & -    & -     &-      & -      & -      &-      &-     & -      & - \\ 
      &       & \star& \star& \star &\star  & -      & -      &-      &-     & -      & - \\ 
      &       & -    & -    & -     &-      & -      & -      &-      &-     & -      & - \\ 
      &       & -    & -    & -     &-      & -      & -      &-      &-     & -      & - \\ 
      &       & -    & -    & -     &-      & \boxed{\star}& -&-      &-     & -      & - \\ 
      &       & -    & -    & -     &-      & -      & -      &-      &-     & -      & - \\ 
      &       & -    & -    & -     &-      & -      & -      &-      &-     & -      & - \\ 
      &       & -    & -    & -     &-      & -      & -      &-      &-     & -      & - \\ 
      &       & -    & -    & -     &-      & -      & -      &-      &-     & -      & - \\ 
      &       & -    & -    & -     &-      & -      & \star  &\star  &\star &\star   & \star \\ 
    \end{array}
  \end{equation*}
  Visually, we have kept $d_i$ in the same place, but have shifted all
  of the earlier numbers to the top of the diagram and all of the
  later ones to the bottom. Notice that in this example $a = 3$.  In
  the right-hand diagram there are visible jumps of size $b=3$ and
  $e = 5$ on either side of the $\boxed{\star}$ in position
  $i$.\end{Example}


\begin{Lemma}
  If $D\colon 0 < ( d_1 = a) < d_2 < \cdots <d_n$ is a degree sequence
  then for all $i\geq 1$
  \begin{equation*}
    \pi_i(D) \geq \pi_i(D^i)
  \end{equation*}
\end{Lemma}
\begin{proof}
  We prove a slightly more general statement.  Let $i\geq 1$ and
  suppose that $D' = \{0,d'_1, \ldots, d'_n\}$ is a degree sequence
  with $d'_i = d_i$.  Then
  \begin{equation*}
    \pi_i(D)
    = \prod_{j\neq i} \frac{d_j}{|d_j-d_i|}
    ,  \ \
    \pi_i(D')
    = \prod_{j\neq i} \frac{d'_j}{|d'_j-d_i|}.
  \end{equation*}
  As all the terms in the product are positive, a sufficient condition
  for $\pi_i(D) \geq \pi_i(D')$ is that
  \begin{equation*}
    \frac{d_j}{|d_j-d_i|} \geq \frac{d_j'}{|d'_j-d_i|}
  \end{equation*}
  for all $j\neq i$.  If $j < i$ then this is equivalent to requiring
  \begin{equation*}
    \frac{d_j}{d_i-d_j} \geq \frac{d'_j}{d_i-d'_j}
    \quad \Longleftrightarrow \quad
    d_j\geq d_j'.
  \end{equation*}
  Conversely, if $j > i$ then the inequality is $d_j\leq d'_j$.  To
  conclude, we simply observe that all of these inequalities hold for
  $D' = D^i$, whence the result follows.
\end{proof}

\noindent We now compute
\begin{align*}
  \pi_i(D^i)
  &= \frac{a(a+1)\cdots(a+(i-2))}{
    (b+1)(b+2)\cdots(b+(i-1))} \frac{(a + b + e + i)\cdots (a + b + e +
    n - 1)}{ (e+1)(e+2)\cdots (e + n - i)}\frac{i!(n-i)!}{n!}{n \choose
    i}\\
  &= \frac{a(a+1)\cdots(a+(i-2))}{
    (b+1)(b+2)\cdots(b+(i-1))} \frac{(n + a + b + e - 1)!e!}  {(a + b +
    e + i - 1)!(n - i + e)!}\frac{i!(n-i)!}{n!}{n \choose i}\\
  &= \frac{(a)\cdots(a+(i-2))}{
    (b+1)\cdots(b+(i-1))} \frac{(n+1) \cdots (n + a + b + e -
    1)}{(i + 1)\cdots (i + a + b + e - 1)} \frac{e!}  {(n-i + 1) \cdots
    (n - i + e)}{n \choose i}.
\end{align*}
\begin{Definition}\label{def:F}
  We define the function $F=F(a,b,e,n,i)$ as the coefficient of
  ${n \choose i}$ in the above computation.  The domain of $F$ is
  $b\geq 0$, $e\geq 0$, $a\geq 2$, $n\geq 3$, $1\leq i\leq n$,
  \begin{equation*}
    F(a,b,e,n,i) = \frac{(a)\cdots(a+(i-2))}{
      (b+1)\cdots(b+(i-1))} \frac{(n+1) \cdots (n + a + b + e -
      1)}{(i + 1)\cdots (i + a + b + e - 1)} \frac{e!}  {(n-i + 1) \cdots
      ((n - i) + e)}.
  \end{equation*}
  In the sequel we will refer to each of the three fractions in the
  above equation as a grouping.  When $i=1$ there are no terms in the
  first grouping.  Similarly, when $e=0$ there are no terms in the
  third grouping.
\end{Definition}
\noindent Our present goal is to show that $F(a,b,e,n,i)$ is at least
$2$ for a suitable range of inputs (e.g.\ $i\leq \lceil n/2\rceil$).
\begin{Lemma}\label{lem:Fa}
  $F$ is increasing as a function of $a$:
  \begin{equation*}
    F(a,b,e,n,i) \leq F(a+1, b,e,n,i).
  \end{equation*}
\end{Lemma}
\begin{proof}
  If $i=1$, then $F(a+1,b,e,n,i)$ is equal to $F(a,b,e,n,i)$ times an
  additional factor which has the form $(s+n+ a)/(s+i + a)$ for some
  $s\in \mathbb{N}$, which is evidently at least 1. If $i>1$, then in
  addition to this extra factor, the numerators of the terms in the
  first grouping in $F(a+1,b,e,n,i)$ will be larger than the
  corresponding terms on the left hand side of the inequality.
\end{proof}

One might hope that $F$ is an increasing function of $n$.  This is not
the case as can be seen in Figure \ref{DesmosGraph}.  However, note
that in the figure $F$ is increasing for $n \geq 40$.  It is no
coincidence that $40 = 2i$ as the following lemma shows.
\begin{figure}[h!]
  \includegraphics[width=.5\textwidth]{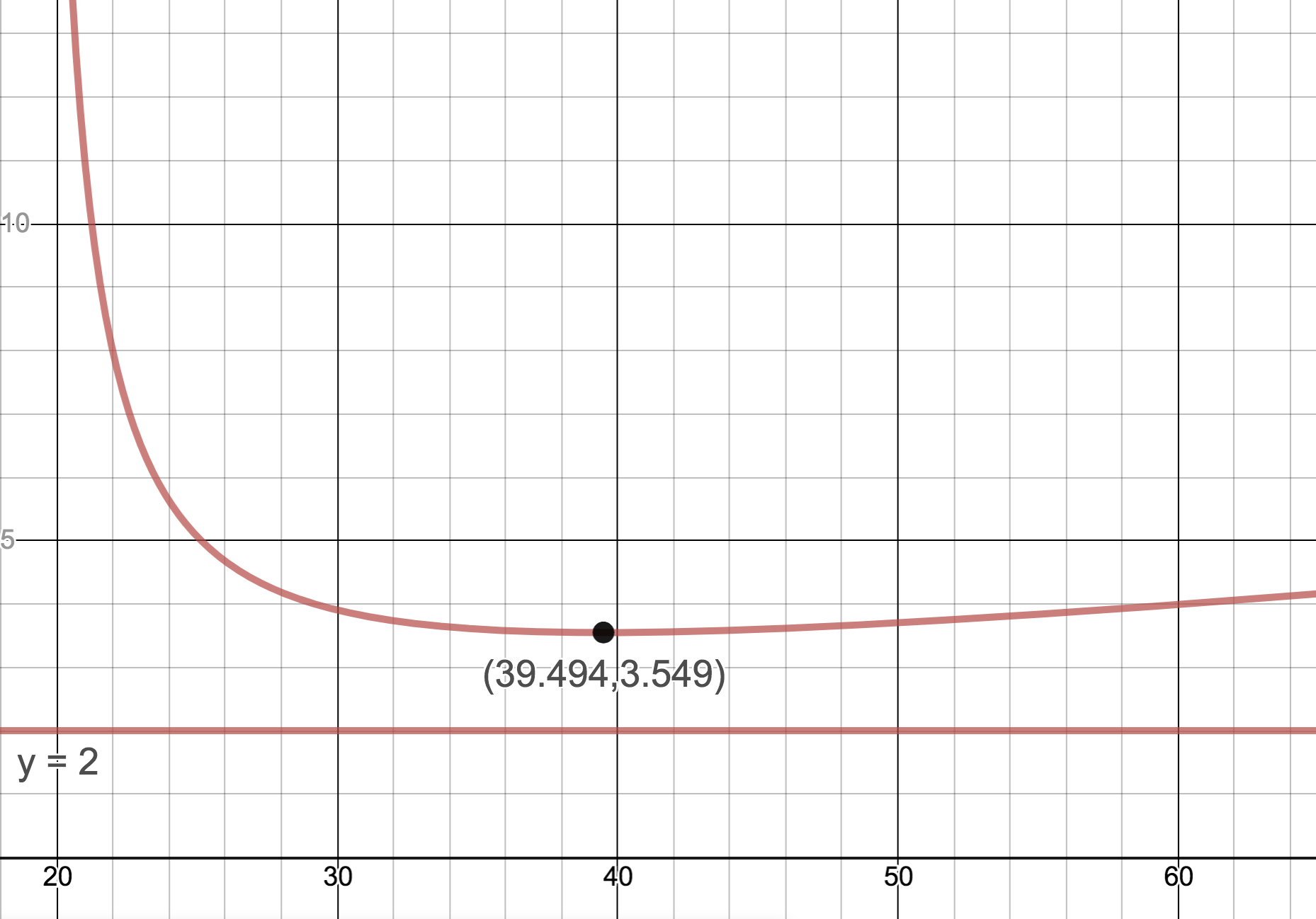}
  \caption{A graph of $F(2,0,1,n,20)$ as a function of $n$.}
  \label{DesmosGraph}
\end{figure}

\begin{Lemma}\label{lem:Fn}
  If $n\geq 2i-1$ and $\reg(D)\leq 2a-2$ then
  \begin{equation*}
    F(a,b,e,n,i) \leq F(a,b,e,n+1, i).
  \end{equation*}
  That is, if $i$ is at most $\frac{n+1}{2}$ then $F$ is an increasing
  function of $n$.
\end{Lemma}
\begin{proof}
  Let $R = \reg(D) = a+b + e - 1$.  Using our assumption on the
  regularity, we have
  \begin{equation*}
    2a-2\geq a+b+e-1
    \quad\Longrightarrow\quad
    a-1\geq b+e.
  \end{equation*}
  This in turn implies
  \begin{equation*}
    R=b+e+a-1\geq 2(b+e)\geq 2e.
  \end{equation*}
  Further if $n \geq 2i-1$ then
  \begin{equation*}
    n \geq (2i-1)\frac{R-e}{R-e}
    = \frac{2iR-R -2ie +e}{R-e}
    = \frac{(i-1)R+e +i(R-2e)}{R-e}
    \geq \frac{(i-1)R +e}{R-e}.
  \end{equation*}
  Finally we compute
  \begin{equation*}
    \frac{F(a,b,e,n+1,i)}{F(a,b,e,n,i)}
    = \frac{(n+a+b+e)(n-i+1)}{(n+1)((n - i) + e + 1)}
  \end{equation*}
  This will be at least $1$ provided
  \begin{equation*}
    (n+a+b+e)(n-i+1)\geq (n+1)((n - i) + e + 1)    
  \end{equation*}
  which is equivalent to:
  \begin{equation*}
    n \geq \frac{(i-1)R+e}{R - e}.
  \end{equation*}
  This is the inequality we have shown above.
\end{proof}
\begin{Remark}
  Notice that Figure \ref{DesmosGraph} shows that we cannot improve
  the bound $n\geq 2i -1$.  Further, note that in this proof we used
  that $\reg(D) \geq 2e$ and that this came from our assumption that
  $\reg(D) \leq 2a-2$.  If we relax that bound, even by one, say to
  $2a-1$ then it will not be true that $F$ is an increasing function
  of $n$.  For instance, consider the following two degree sequences
  (with $a = 2, b = 0, e = 2, i = 3, R = 3$):
  \begin{align*}
    \{0,2,3,4,7,8\}\ ,& \ \{0,2,3,4,7,8,9\}\\
    F(a,b,e,5,i) >& \  F(a,b,e,6,i).
  \end{align*}
\end{Remark}

At this point we present a flowchart that indicates ultimately how we
will prove Theorem \ref{FirstHalfAreDoubleWithCases}.  We have just
seen (Lemmas \ref{lem:Fa} and \ref{lem:Fn}) two crucial observations
about the function $F$.  Using these, a few elementary computations
would allow us to establish Theorem \ref{FirstHalfAreDoubleWithCases}
for the vast majority of degree sequences of pure diagrams.  However,
as mentioned in the introduction, our reduction via Boij-S\"oderberg
theory requires that we consider \emph{all} degree sequences of pure
sub-diagrams of the betti diagram of $M$ and many of these degree
sequences are \emph{not} covered by the lemmas above.

\begin{figure}[h!]
  \centering \includegraphics[width=.8\textwidth]{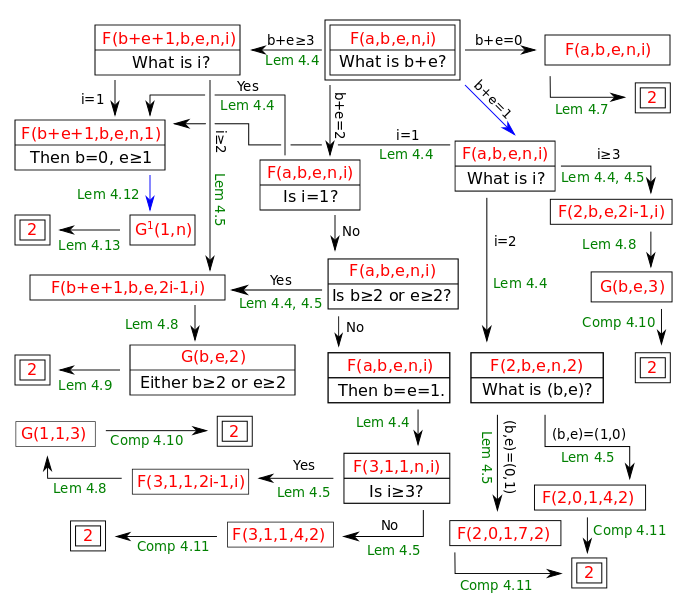}
  \caption{The proof when $n\geq 9$.  The red expression in each box
    is the current lower bound for $F(a,b,e,n,i)$; the black question
    tells one how to proceed.  Arrows are decorated with the possible
    answers to the questions (in black) and the lemma or computation
    used (in green) to obtain the new lower bound (i.e.\ arrows can be
    read as $\geq$ symbols).  The two blue arrows highlight the places
    where our argument differs for $n\in \{6,7,8\}$.}
  \label{fig:FlowN9}
\end{figure}

We begin in the upper right of the chart addressing the case of linear
resolutions. These correspond to the case when $b = e = 0$ and are
handled by the following lemma.
\begin{Lemma}\label{lem:Fbe0}
  If $b = e = 0$, then $F(a,0,0,n,i) \geq 2$.
\end{Lemma}
\begin{proof}
  If $i\geq 2$, then
  \begin{equation*}
    F(a,0,0,n,i)= \frac{a}{1}\frac{(a+1)\cdots(a+(i-2))}{(2)\cdots(i-1)}
    \frac{(n+1) \cdots (n + a -1)}{(i + 1)\cdots (i + a -1)}
    \geq \frac{a}{1}\geq 2.
  \end{equation*}
  On the other hand, if $i=1$ there are no terms in the first
  grouping.  Since $a\geq 2$ and $n\geq 3$, there is at least one term
  in the middle grouping and we have
  \begin{equation*}
    F(a,0,0,n,1)
    =\frac{(n+1)(n+2) \cdots (n + a-1)}{(2)(3)\cdots (a)}
    \geq \frac{(n+1)}{(2)}\geq \frac{4}{2}\geq 2.\qedhere
  \end{equation*}
\end{proof}

\noindent Our approach is now as follows: by Lemma \ref{lem:Fbe0} we
may assume that $b+e \geq 1$.  For fixed $b,e,n,i$ our regularity
assumption provides a minimum possible value of $a$: we have
$a + b + e - 1 \leq 2a - 2$ and thus $a \geq b + e + 1$.  In light of
Lemma \ref{lem:Fa}, it's natural to set $a=b+e+1$.  We can then apply
Lemma \ref{lem:Fn} and decrease $n$ to its minimum possible value of
$n = 2i-1$.  However we will only do this when $i\geq 2$, since we
only want to consider degree sequences with $n\geq 3$; our argument
will need modifications when $i=1$.  Thus, for $i\geq 2$ and
$b+e\geq 1$, we now consider the function $G(b,e,i)$ defined by making
these substitutions.
\begin{align*}
  G(b,e,i)
  &\defeq F(b+e + 1, b, e, 2i-1, i)\\
  &= \frac{(b+e + 1)\cdots(b + e + (i-1))}{(b+1)(b+2)\cdots(b+(i-1))}
    \frac{(2i) \cdots (2i + 2b + 2e - 1)}{(i + 1)\cdots (i + 2b + 2e)}
    \frac{e!}  {(i)\cdots (i + e - 1)}
\end{align*}

We remind the reader that our goal is to find a lower bound for
$\pi_i(D)$ and point out that at this point we have (for $b+e \geq 1$
and $i\geq 2$):
$$\pi_i(D) \geq \pi_i(D^i) \geq F(a,b,e,n,i) \geq G(b,e,i).$$

\begin{Lemma}\label{lem:Gi}
  $G$ is an increasing function of $i$: $G(b,e,i)\leq G(b,e,i+1)$.
\end{Lemma}
\begin{proof}
  We consider the quotient
  \begin{equation*}
    \frac{G(b,e,i+1)}{G(b,e,i)}
    = \frac{b+e + i}{b+i}
    \frac{(2i + 2b + 2e +1)(2i + 2b +2e)(i+1)}{(2i)(2i+1)(i+2b+2e +1)}
    \frac{i}{i+e}.    
  \end{equation*}
  We want this to be at least $1$.  When we cross-multiply and
  subtract we are left with the inequality:
  \begin{multline*}
    4 b^{3} i^{2}+4 b^{2} e i^{2}+4 b e^{2} i^{2}+4 e^{3} i^{2} +4
    b^{2} i^{3}+4 b e i^{3}+4 e^{2} i^{3}+4 b^{3} i+8b^{2} e i
    +8 b e^{2} i+4 e^{3} i\\
    +10 b^{2} i^{2}+14 b e i^{2}+10 e^{2} i^{2}+6 b i^{3}+6 e i^{3} +2
    b^{2} i+2 b e i+2 e^{2} i+2 b i^{2}+2 e i^{2} \geq 0
  \end{multline*}
  which is evident.
\end{proof}

\noindent In consideration of this, since $G(b,e,i) \geq G(b,e,2)$ for
all $i\geq 2$ we show, with a few minor exceptions, that
$G(b,e,2)\geq 2$ for relevant inputs.
\begin{Lemma}\label{lem:Gbe2}
  If either $b\geq 2$ or $e\geq 2$, then $G(b,e,2)\geq 2$.
\end{Lemma}
\begin{proof}
  We simply compute
  \begin{align*}
    G(b,e,2)
    &= \frac{b+e + 1}{b+1}
      \frac{(4) \cdots (2b + 2e +3)}{(3)\cdots (2b + 2e+2)}
      \frac{e!}{(2)\cdots (e+1)}\\
    &= \frac{b+e+1}{b+1}\frac{2b+2e+3}{3}\frac{1}{e+1}.     
  \end{align*}
  This will be at least $2$ if and only if
  \begin{equation*}
    2 b^{2}-2 b e+2 e^{2}-b-e-3\geq 0.
  \end{equation*}
  Now
  \begin{equation*}
    2 b^{2}-2 b e+2 e^{2}-b-e-3 = (b-e)^2 +b^2 + e^2 -b -e -3
  \end{equation*}
  If $b =e$ then this is $2b^2 -2b - 3$ which will be nonnegative
  provided $b \geq 2$.  Otherwise, if either $b$ or $e$ is at least
  $2$ then one of $b^2-b$ or $e^2 - e$ will be at least $2$.  Thus if
  $b\neq e$ then
  \begin{equation*}
    (b-e)^2-3 + (b^2-b) + (e^2 - e) \geq 1 - 3 + 2 = 0.  \qedhere
  \end{equation*}
\end{proof}
\noindent Restricting our attention to the situation where $i\geq 2$,
the lemmas we have established are sufficient to conclude that
$F\geq 2$ for the vast majority of relevant inputs.  The remaining
cases (still assuming that $i\geq 2$) are treated via direct
computation.
\begin{Computation}\label{comp1}
  \begin{equation*}
    G(1,0,3)=2.1
    \qquad
    G(0,1,3)=2.1
    \qquad
    G(1,1,3)=2.4
  \end{equation*}
\end{Computation}
As $G(b,e,i)$ is an increasing function of $i$, these computations
will allow us to obtain the desired lower bound on $F$ when $i\geq 3$.
Indeed, either Lemma \ref{lem:Gbe2} applies or else $b+e =1$ and
$G(b,e,i)\geq G(b,e,3)$ which must be one of the numbers above.

We close with one final computation as well as a discussion of what
happens for $i = 1$. The reader may note that the values of $n$ in
these computations are creeping upwards; this is the first indication
for the hypothesis that $n$ be greater than $9$ in our main theorems.
\begin{Computation}\label{comp2}
  \begin{equation*}
    F(3,1,1,4,2)=2.33 \qquad
    F(2,1,0,4,2)= 2.5 \qquad
    F(2,0,1,7,2)=2
  \end{equation*}
\end{Computation}

\noindent We now close by handling the case $i = 1$.  Note that $i =1$
implies that $b = 0$.  We may assume that $e> 0$ and the assumption
that $\reg(D)\geq 2a-2$ implies that we may assume
$a = b + e+ 1 = e+1$.  What remains is to determine when
$$G^1(e,n) \defeq  F(e+1,0,e,n,1) \geq 2.$$
There is a finite set of inputs for which this lower bound fails, and
these are the source of the 36 betti diagrams of pure modules which
satisfy our regularity bound but to which Theorem
\ref{FirstHalfAreDoubleWithCases} does not apply.

\begin{Lemma}\label{lem:G1e}
  For all $n\geq 3$ and $e\geq 1$, we have
  \begin{equation*}
    F(e+1,0,e,n,1)\leq F(e+2,0,e+1,n,1)
  \end{equation*}
  That is, for all $n$, the function $G^1(e,n)\defeq F(e+1,0,e,n,1)$
  is increasing as a function of $e$.
\end{Lemma}

\begin{proof}
  As usual, we want to establish the following inequality.
  \begin{equation*}
    \frac{F(e+2,0,e+1,n,1)}{F(e+1,0,e,n,1)}
    =\frac{(2e+n+2)(2e+n+1)}{(2e+3)(2e+2)}\cdot
    \frac{(e+1)}{(e+n)}\geq 1
  \end{equation*}
  Cross-multiplying, simplifying, and factoring, we find that this
  equivalent to
  \begin{equation*}
    (n-1)(n-2)(e+1)\geq 0,
  \end{equation*}
  which is evident as $n\geq 3$ and $e\geq 1$.
\end{proof}
\begin{Lemma}\label{lem:G1n9}
  If $n\geq 9$, then $G^1(1,n)\geq 2$.
\end{Lemma}
\begin{proof}
  We compute
  \begin{equation*}
    G^1(1,n)
    =\frac{(n+2)!}{n!}\cdot
    \frac{(n-1)!}{n!}\cdot
    \frac{1}{3}.
  \end{equation*}
  This is greater than $2$ if and only if $n^2-9n+2\geq 0$, which is
  the case for $n$ at least $9$.
\end{proof}
As before, some sporadic cases will be handled by a few direct
computations.
\begin{Computation}\label{comp3}
  \begin{equation*}
    F(3,1,1,6,2)=4.2\qquad G^1(2,6)=2.
  \end{equation*}
\end{Computation}
We have need of one final computation that will reduce from infinite
to finite the number of degree sequences of pure diagrams that do not
satisfy the hypotheses of our theorem.  Indeed, if the regularity
bound is strengthened by one and we assume that
$\reg(D)\leq 2a-\mathbf{3}$, then the minimum possible value of $a$ is
$b+e+2$.  We compute:
\begin{Computation}\label{comp4}
  For $i\in\{1,2,3,4\}$ and $(b,e)\in\{(1,0),(0,1)\}$, we have
  $F(3,b,e,6,i)\geq 2$.
\end{Computation}

We are now ready to put the jigsaw puzzle together and prove Theorem
\ref{FirstHalfAreDoubleWithCases}.  For the reader's convenience, we
have restated it below in an equivalent form.
\begin{Proposition}\label{FirstHalfAreDoubleEquivalent}
  Let $D$ be a degree sequence with $\reg(D)\leq 2a-2$ and $n\geq 9$.
  Then for each $1\leq i\leq \lceil n/2\rceil$,
  $\pi_i(D)\geq 2{n \choose i}$.  If $n\in\{6,7,8\}$ and either
  $a\neq 2$ or $b+e \neq 1$, then the same conclusion holds.
\end{Proposition}

\begin{figure}[h!]
  \centering \includegraphics[width=.8\textwidth]{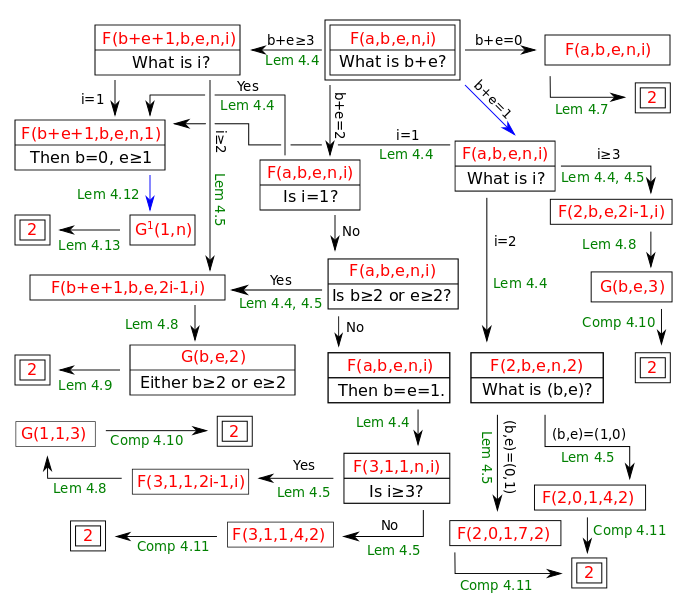}
  \caption{The proof when $n\geq 9$.  The red expression in each box
    is the current lower bound for $F(a,b,e,n,i)$; the black question
    tells one how to proceed.  Arrows are decorated with the possible
    answers to the questions (in black) and the lemma or computation
    used (in green) to obtain the new lower bound (i.e.\ arrows can be
    read as $\geq$ symbols).  The two blue arrows highlight the places
    where our argument differs for $n\in \{6,7,8\}$.}
\end{figure}

\begin{proof}[Proof of Proposition \ref{FirstHalfAreDoubleEquivalent}]
  The proof amounts to piecing together the lemmas and computations
  above and is depicted in the flowchart (Figure \ref{fig:FlowN9}).  A
  key point is that for a fixed degree sequence $D$, while $D^i$ (and
  the associated nonlinear parts $b$ and $e$) depends on the value of
  $i$, the sum $b+e$ of $D^i$ is a function only of the original
  degree sequence $D$ and not of $i$.  For $n\geq 9$, refer to the
  flow chart.

  If the resolution is linear so that $b+e=0$, then Lemma
  \ref{lem:Fbe0} applies to give the desired conclusion.  If
  $b+e\geq 3$, then we apply Lemma \ref{lem:Fa} and decrease $a$ to
  its minimum possible value while maintaining our regularity
  assumption.  Then, if $i\geq 2$, we apply Lemma \ref{lem:Fn},
  decreasing $n$ to get
  \begin{equation*}
    F(a,b,e,n,i)\geq F(b+e+1,b,e,n,i)\geq F(b+e+1,b,e,2i-1,i)=G(b,e,i).
  \end{equation*}
  Since $b+e\geq 3$, either $b\geq 2$ or $e\geq 2$ regardless of the
  value of $i$.  Thus, in all cases we may apply Lemma \ref{lem:Gi}
  decreasing the value of $i$ and then apply Lemma \ref{lem:Gbe2} to
  conclude
  \begin{equation*}
    F(a,b,e,n,i)\geq G(b,e,i)\geq G(b,e,2)\geq 2.
  \end{equation*}
  If $i=1$, we still apply Lemma \ref{lem:Fa}.  Then we note that this
  implies $b=0$.  Now Lemmas \ref{lem:G1e} and \ref{lem:G1n9} allows
  us to conclude
  \begin{equation*}
    F(a,b,e,n,i)
    \geq F(b+e+1,b,e,n,i)=F(e+1,0,e,n,1)=G^1(e,n)
    \geq G^1(1,n)
    \geq 2.
  \end{equation*}
  Now if $b+e=2$, the above argument fails only for those values of
  $i$ where $b=e=1$ (because Lemma \ref{lem:Gbe2} fails); when $i=1$,
  the argument needs no modification.  If $b=e=1$ and $i\geq 3$, then
  we apply Lemmas \ref{lem:Fa}, \ref{lem:Fn}, and \ref{lem:Gi} just as
  above only this time we use Computation \ref{comp1} to conlcude
  \begin{equation}\label{eq:n9be2}
    F(a,1,1,n,i)\geq F(3,1,1,n,i)\geq F(3,1,1,2i-1,i)=G(1,1,i)\geq G(1,1,3)>2.
  \end{equation}
  If $i=2$, then rather than decreasing $n$ to $2i-1=3$ in applying
  Lemma \ref{lem:Fn}, we set $n=4$ and use Computation \ref{comp2}.
  \begin{equation*}
    F(a,1,1,n,2)\geq F(3,1,1,n,2)\geq F(3,1,1,4,2)\geq 2.
  \end{equation*}
  If $b+e=1$, the chain of inequalities \eqref{eq:n9be2} still holds
  for $i\geq 3$ and the logic from above still applies for $i=1$.
  Thus, the only remaining case is $i=2$ and our assumptions imply
  $(b,e)\in\{(0,1),(1,0)\}$.  When $(b,e)=(0,1)$ (resp.\
  $(b,e)=(1,0)$), apply Lemma \ref{lem:Fn} to decrease $n$ to 7
  (resp. 4), then apply Computation \ref{comp2} to get
  \begin{equation*}
    F(a,b,e,n,2)\geq F(2,b,e,n,2)\geq 2.
  \end{equation*}
  If $n\in \{6,7,8\}$, the proof differs only in a few places and
  these are depicted in the flow chart by two blue arrows.  The arrow
  on the left hand side concerns the setting where $b+e\geq 3$ and
  $i=1$, which implies that $b=0$ and $e\geq 2$.  This time we apply
  Lemma \ref{lem:Fn} and decrease $n$ to the value of $6$, then apply
  Lemma \ref{lem:G1e} setting $e=2$ and use Computation \ref{comp3}
  \begin{equation*}
    F(a,b,e,n,i)\geq F(e+1,0,e,n,1)\geq F(e+1,0,e,6,1)=G^1(e,6)\geq G^1(2,6)=2
  \end{equation*}
  The second blue arrow concerns the case that $b+e=1$, and for
  finitely many degree sequences, our method fails here.  If
  $\reg(D)\leq 2a-3$, then we apply Lemma \ref{lem:Fa} decreasing $a$
  to the minimum possible value of $a=b+e+1=3$.  Next apply Lemma
  \ref{lem:Fn} and set $n=6$.  Noting that $(b,e)\in\{(1,0),(0,1)\}$,
  we use computation \ref{comp4} to obtain
  \begin{equation*}
    F(a,b,e,n,i)
    \geq F(3,b,e,n,i)
    \geq F(3,b,e,6,i)\geq 2.\qedhere
  \end{equation*}
\end{proof}


\section*{Acknowledgments}
\noindent We thank Daniel Erman for inspiring this project as well as
for the many conversations about Boij-S\"oderberg theory over the
years.  We thank Craig Huneke for the suggestion to look at how the
sum of the betti numbers behaves with respect to these
Boij-S\"oderberg decompositions.  A portion of this research was
conducted at the Fields Institute and the second author thanks them
for their hospitality during that period.  Finally, we are grateful
for helpful conversations with David Eisenbud, Srikanth Iyengar,
Anurag Singh, and Mark Walker.

\bibliographystyle{plain} \bibliography{References}
\end{document}